\documentclass{amsart}

\usepackage{amsmath, amsthm, amssymb}
\usepackage{epsfig}
\usepackage{array}
\usepackage{amsmath}
\usepackage{amsfonts}
\usepackage{amsmath}%
\usepackage{amsfonts}%
\usepackage{amssymb}%
\usepackage{graphicx}

\usepackage{amsfonts}
\usepackage{amssymb}
\usepackage{amsmath,amsxtra,amssymb}

\numberwithin{equation}{section}

\newtheorem{theorem}{Theorem}[section]
\newtheorem{lemma}[theorem]{Lemma}

\newtheorem{corollary}[theorem]{Corollary}
\newtheorem{proposition}[theorem]{Proposition}

\newcommand{\C}{{\mathbb{C}}}

\newcommand{\ra}{\rightarrow}
\newcommand{\id}{{\iota}}

\newcommand{\om}{{\omega}}
\newcommand{\tp}{{\widehat{\otimes}}}
\newcommand{\vtp}{{\,\overline{\otimes}\,}}
\newcommand{\botimes}{\vtp}

\newcommand{\h}{{\mathcal H}}

\newcommand{\fee}{{\varphi}}

\newcommand{\LL}{{\mathcal{L}^{\infty}(\G)}}
\newcommand{\LO}{{\mathcal{L}^{1}(\G)}}
\newcommand{\LT}{{\mathcal{L}^{2}(\G)}}
\newcommand{\MG}{{\mathcal{M}(\G)}}
\newcommand{\CG}{{\mathcal{C}(\G)}}

\newcommand{\PG}{{\mathcal{P}(\G)}}

\newcommand{\CZ}{{\mathcal{C}_0(\G)}}

\newcommand{\G}{\mathbb G}

\newcommand{\g}{\mathbb G}

\def\proclaim #1. #2\par{\medbreak
\noindent{\bf#1.\enspace}{\sl#2}\par\medbreak}
\allowdisplaybreaks

\title
{Harmonic Operators of Ergodic Quantum Group Actions}
\author[M. Amini, M. Kalantar \and M. S. M. Moakhar]{Massoud Amini \and Mehrdad Kalantar \and Mohammad S. M. Moakhar}

\address{Department of Mathematics,
        Tarbiat Modares University, Tehran 14115-134, Iran}
\email{mamini@modares.ac.ir}

\address{ School of Mathematics and Statistics,
        Carleton University, Ottawa, Ontario, Canada K1S 5B6}
\email{mkalanta@math.carleton.ca}

\address{Department of Mathematics,
        Tarbiat Modares University, Tehran 14115-134, Iran}
\email{m.mojahedi@modares.ac.ir}


\begin{document}

\begin{abstract}
In this paper we study the harmonic elements of (convolution)
Markov maps associated to (ergodic) actions of locally compact quantum groups on ($\sigma$-finite) von Neumann algebras.
We give several equivalent conditions under which
the harmonic elements are trivial.
\end{abstract}

\maketitle

\section{Introduction and Preliminaries}\label{sect1}

The noncommutative Poisson boundary, i.e., the space of fixed points of
normal unital completely positive maps on von Neumann algebras,
has been extensively studied in the last few decades, both in general
and concrete settings \cite{Choi-Eff}, \cite{I}, \cite{Arv04}, \cite{II}, \cite{Arv07}.
These boundaries are of particular interest for the class of Markov maps associated to (homogeneous)
Markov processes on algebraic structures, like the convolution maps of probability
measures on locally compact (quantum) groups.

In \cite{I} Izumi introduced noncommutative random walks on discrete quantum groups
and studied the associated Poisson boundaries.
In this paper we consider a more general setup. The Markov operators whose
Poisson boundaries are investigated here arise from the convolution
action of quantum probability measures on locally compact quantum groups
on von Neumann algebras.

To further clarify the setup and motivate the discussion, let us consider the classical case:
let $G$ be a locally compact group, acting measurably on a $\sigma$-finite probability space $(X, \nu)$.
When $\nu$ is quasi-invariant, $G$ acts canonically on $L^\infty(X, \nu)$, and
for a probability measure $\mu$ on $G$ one may define
a convolution map on $L^\infty(X, \nu)$ by
\[
\mu \star h (x) \,:=\, \int_X\, h(gx)\, d\mu(g) \hspace{1cm} h\in L^\infty(X, \nu)\,,\ \ x\in X\,.
\]
This convolution map defines a Markov operator, whence a (time independent) Markov chain on $X$
with transition probabilities
\[
p(x, A) \,=\, \mu\left(\{\,g\in G : gx \in A\,\}\right) \hspace{1cm} x\in X \,,\, A \subseteq X\,.
\]
A function $h\in L^\infty(X, \nu)$ is called $\mu$-harmonic if it is fixed under convolution by $\mu$, i.e.,
\[
\int_X\, h(gx)\, d\mu(g) \,=\, h(x) \hspace{1cm} \text{for} \,\,\, \nu-a.e. \,\,\, x\in X\,.
\]
The space of all $\mu$-harmonic functions is isomorphic, through the Poisson transformation, to the
space of all essentially bounded measurable functions on the Poisson boundary of the associated Markov chain on $X$ \cite{Kai92}.

In this paper we investigate harmonic elements of the convolution maps associated to quantum group actions. 
Convolution Markov operators on $G$-spaces are studied in the classical setting (cf. \cite{Kai-Woe}, \cite{FR}, \cite{GR}),
but some of our results here on harmonic functions are new even for group actions.
A more detailed study of the classical convolution Markov chains on transformation groups
is done in \cite{K-Moj}.

\par

Let us recall basic definitions and set the terminology used in this paper.
For more details on locally compact quantum groups and their actions we refer the reader to \cite{KV} and \cite{V}.

A {\it locally compact quantum group} $\G$ is a
quadruple $(\LL, \Gamma, \varphi, \psi)$, where $\LL$ is a
von Neumann algebra,
$\Gamma: \LL\to \LL \vtp \LL$
is a co-associative co-multiplication,
and $\varphi$ and  $\psi$ are the (normal faithful semi-finite) left and right
Haar weights on $\LL$, respectively.
The corresponding GNS Hilbert spaces
$L^2(\G,\varphi)$ and $L^2(\G,\psi)$ are isomorphic and are denoted by the same notation $\LT$.
The (left) fundamental unitary $W$ of $\G$ is a unitary
operator on $\LT \otimes \LT$, satisfying the pentagonal equation that implements
the co-multiplication $\Gamma$.

The \emph{reduced quantum group $C^*$-algebra}
$\CZ = \overline{\{\,(\id\otimes\om)\,W \,:\, \om\,\in\, B(H)_*\,\}}^{\|\cdot\|}$
is a weak$^*$ dense $C^*$-subalgebra of $\LL$.
Let $\LO$ be the predual of $\LL$. The pre-adjoint of
$\Gamma$ induces  an associative completely contractive multiplication
\begin{equation}\label {F.mul}
\star \, :\,  f \otimes g \,\in\,
\LO\,\tp\, \LO\, \to \,f \star g\, =\, (f \otimes g)\, \Gamma\, \in\, \LO
\end{equation}
on $\LO$. Also for the operator dual $\MG$ of $C_{0}(\G)$, there exists  a completely contractive multiplication on $\MG$ given by
the convolution
\[
\star \,:\, \mu\otimes \nu\, \in\, \MG\,\tp \, \MG \,\mapsto\, \mu \star \nu
\,:=\, \mu\, (\id\otimes \nu)\,\Gamma \,=\, \nu \,(\mu \otimes \id)\,\Gamma
\,\in \,\MG
\]
such that  $\MG$ contains $\LO$ as a norm closed two-sided ideal. In particular, every $\mu\in\MG$ induces a normal completely bounded map $\Phi^\mu$
on $\LL$. When $\mu$ is a state, $\Phi^\mu$ is a \emph{Markov map}, i.e., it is also unital and completely positive.



For a von Neumann algebra $N$, a (left) action $\alpha : \G\curvearrowright N$
of $\G$ on $N$ is an injective $*$-homomorphism $\alpha: N \ra \LL\vtp N$ such that
\[
(\Gamma\otimes\id)\,\alpha\,=\,(\id\otimes\alpha)\,\alpha\,.
\]
The action $\alpha$ is called {\it ergodic} if the fixed point algebra is trivial, that is,
\[
N^\alpha\,:=\,\{\,x\,\in\,N\,:\,\alpha(x)\,=\,1\otimes x\,\} \,=\, \C 1\,.
\]

 \section{Convolution operators on $\G$-spaces}

For $\mu\in \MG$ by \cite[Theorem 2.1]{K-IJM2} one obtains a (unique) normal completely bounded map
$\Phi^\mu_\alpha$ on $N$ such that
\begin{equation}\label{a1}
\alpha\,\Phi_\alpha^\mu\, =\, (\Phi^\mu\otimes\iota)\,\alpha\,.
\end{equation}
When $\mu \in \LO$, we have $\Phi_\alpha^\mu\,=\,(f\otimes\id)\,\alpha$.

We denote by $\PG$ the set of all states $\mu$ in $\MG$.
For $\mu\in\PG$, the corresponding operator $\Phi_\alpha^\mu$
is a Markov map on $N$.
We denote by $\mathcal{H}^\mu_\alpha$ the space of $\mu$-harmonic elements in $N$,
i.e., the space of fixed points of the Markov operator $\Phi_\alpha^\mu$.
This is a weak$^*$ closed operator system in $N$,
and if we choose a free ultrafilter $\mathcal{U}$ on $\mathbb N$, we obtain an
unital completely positive idempotent map $E_\alpha^\mu(x)$ on $N$,
\begin{equation}\label{new1}
E_\alpha^\mu(x) \,=\, \text{weak*}-\lim_{\mathcal{U}}\,{\frac{1}{n}\,\sum_{k=1}^n\, {(\Phi_\alpha^\mu)^k (x)}}\,.
\end{equation}
This allows us to define a von Neumann algebra structure on the operator system $\mathcal{H}^\mu_\alpha$
via the \emph{Choi--Effros} product $x\circ y \,:=\, E_\alpha^\mu(xy)$ (cf. \cite{Choi-Eff}).



One may consider the co-multiplication $\Gamma$ as a (left) action of $\G$ on $\LL$.
The following result describes $\h_\alpha^\mu$ for
$\alpha : \G \curvearrowright N$ in terms of $\h_\Gamma^\mu$.


\begin{proposition}\label{thm1}
Let $\alpha : \G \curvearrowright N$ be an action of a locally compact quantum group $\G$ on a von Neumann algebra $N$
and $\mu\in \PG$. Then
\[
\mathcal{H}_\alpha^\mu \,=\, \{\,x\in N \,:\, \alpha(x)\,\in\, \mathcal{H}_\Gamma^\mu \,\botimes\, N\,\}\,.
\]
\end{proposition}
\begin{proof}
The inclusion $ \supseteq $ follows from $(\ref{a1})$ and injectivity of $\alpha$.
For the inverse inclusion, suppose that $x\in \mathcal{H}_\alpha^\mu$, i.e., $\Phi_\alpha^\mu(x) = x$.
Then using $(\ref{a1})$ we obtain
\[
\Phi^\mu\, \big((\iota\otimes \om)\,\alpha(x)\big)\, =\,
(\iota\otimes \om)\,\big((\Phi^\mu\otimes\iota)\,\alpha(x)\big)\, =\,
(\iota\otimes \om)\, \big(\alpha\,(\Phi_\alpha^\mu(x))\big)\, =\,
(\iota\otimes \om)\,\alpha(x)
\]
for all $\om\in N_*$.
Hence $(\iota\otimes \om)\alpha(x)\in \mathcal{H}_\Gamma^\mu$ for all $\om\in N_*$ and therefore
$\alpha(x)\in \mathcal{H}_\Gamma^\mu\otimes_{\mathcal F} N$,
where $\otimes_{\mathcal F}$ denotes the Fubini tensor product (cf. \cite{ERbook}).
Since $\mathcal{H}_\Gamma^\mu$ is a von Neumann algebra with its Choi--Effros product,
it follows from \cite[Proposition 3.3]{Ruan1992} that
$\mathcal{H}_\Gamma^\mu\otimes_{\mathcal F} N = \mathcal{H}_\Gamma^\mu\vtp N$,
and the inclusion follows.
\end{proof}

\begin{corollary}\label{cor}
Let $\alpha : \G \curvearrowright N$ be an ergodic action of a locally compact quantum group $\G$ on a von Neumann algebra $N$
and $\mu\in \PG$. If $\mathcal{H}_\Gamma^\mu = \C 1$, then $\mathcal{H}_\alpha^\mu = \C1$.
\end{corollary}
\begin{proof}
Let $x\in \mathcal{H}_\alpha^\mu$. By the above proposition,
there exists $y\in N$ such that $\alpha(x) = 1\otimes y$.
Therefore
\[
1\otimes \alpha(y)\, = \, (\id \otimes \alpha)\,\alpha(x) \, = \, (\Gamma \otimes \id)\,\alpha(x) \, = \,
(\Gamma \otimes \id)\,(1\otimes y) \, = \, 1\otimes 1\otimes y\,,
\]
which implies $y \in \C 1$, by ergodicity of $\alpha$, hence $x\in \C1$.
\end{proof}

For $\om\in N_{*}$, we define a completely bounded map $\phi_\om: N\rightarrow \LL$
by $\phi_\om (x) \,=\, (\iota\otimes \om) \, \alpha(x)$.
Then the following are easily derived
\begin{equation}\label{new2}
(\iota\otimes \phi_\om)\,\alpha \,=\, \Gamma\circ\phi_\om
\hspace{1cm} \text{and} \hspace{1cm}
\Phi^\mu\circ\phi_\om \,=\, \phi_{\mu\star\om}\,,
\end{equation}
for all $\om\in N_*$ and $\mu\in\MG$, where $\mu\star\om := (\Phi_\alpha^\mu)_*(\om)$.

An essential fact behind many of the results concerning the convolution maps in the quantum setting
is that 
for any $\mu\in \PG$, the convolution map $\Phi^\mu$ is a faithful map on $\LL$.
This follows easily from the faithfulness of the Haar weight,
and its invariance under convolution maps (c.f. \cite[Lemma 3.4]{KNR}).
This, together with $(\ref{a1})$, yield that for any action $\alpha : \G \curvearrowright N$
and $\mu\in\PG$, the map $\Phi_\alpha^\mu$ is faithful on $N$.
%
%
But, in the absence of a ``Haar weight'' for a general action $\alpha : \G \curvearrowright N$,
the analogous result for the map $\phi_\om$ is by no means trivial.


\begin{theorem}\label{lem}
Let $\alpha : \G \curvearrowright N$ be an action of a locally compact quantum group $\G$ on a
$\sigma$-finite von Neumann algebra $N$. Then the following are equivalent:
\begin{itemize}
\item [1.]
the action $\alpha$ is ergodic;
\item[2.]
for any nonzero $\om\in N_{*}^+$, the map $\phi_\om \,:\, N \longrightarrow \, \LL$ is faithful.
\end{itemize}
\end{theorem}
\begin{proof}
$(2) \Rightarrow (1)$:\ Suppose
$\alpha$ is not ergodic. Let $e \in N^\alpha$ be a non-trivial projection,
and $\om\in N_{*}$ a normal state whose support projection is $e$.
Then we have
\[
\phi_\om(1-e)\, =\, (\id\otimes\om)\,\alpha(1-e) \,=\, \langle\, \om\,,\,1-e\,\rangle\, 1\, =\, 0\,.
\]
$(1) \Rightarrow (2)$: \
Let $\om\in N_{*}^+$ be non-zero, and suppose $x\in N^+$ is such that $\phi_\om (x) = 0$.
Denote by $p$ the support projection of $\om$. Then for every $\rho\in \LO$ we have
$\om((\rho \otimes\iota)\alpha(x)) = 0$. This implies
$(\rho \otimes\iota)(\alpha(x)(1\otimes p)) = 0$ for all $\rho\in \LO$, and therefore
$\alpha(x)(1\otimes p) = 0$.
Now consider the left ideal
\[
M \,:=\, \{\,y\in N \,:\, \alpha(y)\,(1\otimes p) \,=\, 0\, \}\,,
\]
and suppose $q\in N$ is a projection such that $M = Nq$.
Then for every $\rho\in \LO$ we obtain
\[
\alpha\,((\rho \otimes\iota)\,\alpha(q))\,(1 \otimes p) \,=\,
(\Phi^\rho\otimes\iota)\,\left(\alpha(q)\,(1\otimes p)\right) \,=\, 0\,,
\]
which shows that $(\rho \otimes\iota) \alpha(q) \in M$, whence $[(\rho \otimes\iota) \alpha(q)] q =  (\rho \otimes\iota) \alpha(q)$
for all $\rho\in \LO$.
This implies $\alpha(q)(1\otimes q) = \alpha(q)$, i.e.
\begin{equation}\label{032}
\alpha(q)\leq 1\otimes q\,.
\end{equation}
Denote $q' = 1 - q$, and let $\theta$ be a normal faithful state on $N$ with the $GNS$ map
$\Lambda_\theta : N \rightarrow H_\theta$. Let
$\tilde\theta$ be the dual weight on the crossed product $\G\, {}_\alpha\!{\ltimes} \,N$ (cf. \cite[Definition 3.1]{V}),
with modular operator and modular conjugation $\nabla_{\tilde\theta}$ and $J_{\tilde\theta}$. Let
and $\hat\nabla$ and $\hat J$ be the modular operator and modular conjugation of the Haar weight $\hat\fee$ of the dual quantum group $\hat \G$.
Then, 
for every $y\in N$ and $\xi \in \mathcal{D}({\hat\nabla}^{\frac12})$ 
we have
\[  \begin{array}{cll}
\alpha(q')\,(1\otimes q)\,(\hat J\, {\hat \nabla}^{\frac 12}\,\xi\,\otimes\,\Lambda_\theta(y))
& = & \alpha(q')\,(\hat J\, {\hat \nabla}^{\frac 12}\,\xi\,\otimes\,\Lambda_\theta(qy))\\
& = & J_{\tilde\theta}\, \nabla_{\tilde\theta}^{\frac 12}\,\alpha(y^*q)\,(\xi\otimes\Lambda_\theta(q'))
\hspace{0.5cm} \text{\bigg(\,by \cite[Lemma 3.11]{V}\,\bigg)}\\
& = & J_{\tilde\theta}\, \nabla_{\tilde\theta}^{\frac 12}\,\alpha(y^*)\,\alpha(q)\,(1\otimes q')\,
(\xi\otimes\Lambda_\theta(q'))\\
& = & 0 \hspace{2.7cm} \text{\bigg(\,since $\alpha(q)(1\otimes q')=0$ by (\ref{032})\,\bigg)}\,.\\
\end{array} \]
Hence $\alpha(q')(1\otimes q)=0$, which then together with
(\ref{032}) implies that
$\alpha(q)= 1\otimes q$. Since $\alpha$ is ergodic, we conclude that $q=0$ or $q=1$. 
But $q=1$ 
implies $\om = 0$, so it follows that $q = 0$, and in particular $x=0$.
\end{proof}

Now, using the above theorem, we can prove a quantum version of the ``\emph{Maximum Principle}'',
which clears the path to generalize a number of results on the triviality of certain classes of harmonic functions.
For this, we need to impose a (rather weak) restriction on $\mu$.
We prove a Maximum Principle for \emph{spread-out} quantum probability measures.
The measure $\mu\in \PG$ is called spread out if there are $n\in\mathbb N$ and $0\neq \om \in\LO^+$ such that
$\om\leq\mu^n$, or equivalently, if $\mu^n = \om_a + \om_s$, for some $n\in\mathbb N$,
where $0\neq \om_a\in\LO^+$ and $\om_s\in\MG^+$.

Note that this is not a very restrictive condition. In the classical
setting of locally compact groups, in order to have a well-defined notion of
measure-theoretical boundaries, one has to restrict to such probability measures.
For discrete (quantum) groups, every (quantum) measure is spread-out. Also the assumption that $\mu$ is spread-out
is not necessary in the following lemma (and results after that) if $\G$ is co-amenable (which includes the case of locally compact groups).
On the other hand, such a restriction is expected as we work in
the very general setting of measurable actions and do not impose any continuity condition on the action.

\begin{lemma} [The Maximum Principle] \label{lemma}
Let $\alpha : \G \curvearrowright N$ be an ergodic action of a locally compact quantum group $\g$ on a $\sigma$-finite
von Neumann algebra $N$, and let $\mu\in\PG$ be a non-degenerate spread-out state.
Suppose that $x\in \h_\alpha^\mu$ is a
self-adjonit element which attains its norm on $N_{*,1}^+$. Then $x\in \C1$.
\end{lemma}

\begin{proof}
Suppose that $0\leq x\in  \h_\alpha^\mu$, $\|x\| = 1$,
and $\nu$ is a normal state on $N$ such that $\langle \nu,x\rangle = 1$.
Let $\mu\in\PG$ be non-degenerate and spread-out.
Define $\displaystyle\rho := \sum_{n=1}^\infty \frac{1}{2^n} \mu^n \in \PG$.
Since $\mu$ is spread-out, we have $\rho = \om_a + \om_s$, where $0\neq \om_a\in\LO^+$ and $\om_s\in\MG^+$.
Moreover, since $\mu$ is non-degenerate, $\rho$ is faithful on $\CZ$ and extends uniquely to a strictly continuous faithful state on the multiplier $C^*$-algebra $\CG$ of $\CZ$.
Therefore $\om_a\star\rho \in \LO^+$ is faithful on $\LL$, and
\begin{eqnarray*}
0\,\leq\,
\langle\, \om_a\star\rho\,,\, \phi_\nu(1-x)\,\rangle &=&
\langle\, \nu\,,\, \Phi_\alpha^{\om_a\star\rho}(1-x) \,\rangle \,=\,
\langle\, \nu\,,\, \Phi_\alpha^{\om_a}\, \Phi_\alpha^{\rho}(1-x) \,\rangle\\
&=&
\langle\, \nu\, ,\, \Phi_\alpha^{\om_a}(1-x)\, \rangle
\,\leq\,
\langle\, \nu \,,\, \Phi_\alpha^{\rho}(1-x)\, \rangle\\
&=& \langle\, \nu\, ,\, 1-x\, \rangle \,=\, 0\,.
\end{eqnarray*}
By Theorem \ref{lem}, the map $\phi_\nu$ is faithful, hence $1-x = 0$.
%
\end{proof}

\begin{proposition}\label{thm3}
Let $\alpha : \G \curvearrowright N$ be an ergodic action of a locally compact quantum group $\g$ on a $\sigma$-finite
von Neumann algebra $N$, and let $\mu\in\PG$ be non-degenerate and spread-out.
Then the following are equivalent:
\begin{itemize}
\item[1.]
every normal state on $\mathcal{H}_\alpha^\mu$ can be extended to a normal state on $N$;
\item[2.]
$\mathcal{H}_\alpha^\mu \,=\, \C1$.
\end{itemize}
\end{proposition}

\begin{proof}
We just need to prove $(1) \Rightarrow (2)$. Consider $\mathcal{H}_\alpha^\mu$ with its
von Neumann algebra structure, and let $p \in \mathcal{H}_\alpha^\mu$ be a projection. Then $p$ attains its norm on a
normal state on $\mathcal{H}_\alpha^\mu$, which, by the assumption, can be extended to a normal state on $N$.
Hence $p\in\C1$, by Lemma \ref{lemma}, therefore $\mathcal{H}_\alpha^\mu \,=\, \C1$.
\end{proof}
The following corollaries are then immediate.

\begin{corollary}\label{cor1}
Let $\alpha : \G \curvearrowright N$ be an ergodic action of a locally compact quantum group $\g$ on a $\sigma$-finite
von Neumann algebra $N$, and let $\mu\in\PG$ be non-degenerate and spread-out.
If the map $E_\alpha^\mu:N \rightarrow \mathcal{H}_\alpha^\mu$ is normal, then
$\mathcal{H}_\alpha^\mu \,=\, \C1$.
\end{corollary}

\begin{corollary}
Let $\alpha : \G \curvearrowright N$ be an ergodic action of a locally compact quantum group $\g$ on a $\sigma$-finite
von Neumann algebra $N$, and let $\mu\in\PG$ be non-degenerate and spread-out. If
$\mathcal{H}_\alpha^\mu$ is a subalgebra of $N$, then
$\mathcal{H}_\alpha^\mu \,=\, \C1$.
\end{corollary}


We close the section by proving a result on the multiplicative structure of $\h_\alpha^\mu$.
The following proposition shows that $\mathcal{H}_\alpha^\mu$
can not contain any non-trivial $*$-subalgebra of $N$.
\begin{proposition}\label{thm4}
Let $\alpha : \G \curvearrowright N$ be an ergodic action of a locally compact quantum group $\g$ on a $\sigma$-finite
von Neumann algebra $N$, and let $\mu\in\PG$ be non-degenerate and spread-out.
If $x\in \mathcal{H}_\alpha^\mu$ is such that $xx^*,x^*x\in \mathcal{H}_\alpha^\mu$, then $x\in\C1$.
\end{proposition}
\begin{proof}
By the assumption we have $\Phi_\alpha^\mu(x^*x) = x^*x = \Phi_\alpha^\mu(x^*)\Phi_\alpha^\mu(x)$.
Hence, by \cite[Corollary 5.2.2]{ERbook},
$\Phi_\alpha^\mu(yx) = \Phi_\alpha^\mu(y) \Phi_\alpha^\mu(x)$,
for all $y\in N$. 
Similarly, we have $yx^*\in\mathcal{H}_\alpha^\mu$, for all $y\in \mathcal{H}_\alpha^\mu$.
Hence, the weak$^*$ closed subalgebra generated by $1, x, x^*$ (denoted by $\mathcal{A}$)
in $N$ is contained in $\mathcal{H}_\alpha^\mu$.
Now, let $p\in \mathcal{A}$ be a nonzero projection
(in the von Neumann algebra structure of $\mathcal{A}$ inherited from $N$),
since $p$ attains its norm on a normal state on $N$, it follows from Lemma \ref{lemma} that
$p=1$, therefore $\mathcal{A}=\C1$. In particular, $x\in \C1$.
\end{proof}


\section{Triviality of certain classes of harmonic elements}

Having the Maximum Principle at our disposal, we can prove
the following result, similar to \cite[Theorem 3.7]{KNR}.
We denote by $\mathcal{K}(H)$ the space of all compact operators on a Hilbert space $H$.

\begin{theorem}\label{thm5}
Let $\alpha : \G \curvearrowright N$ be an ergodic action of a locally compact quantum group $\g$ on a $\sigma$-finite
von Neumann algebra $N$. If $\mu\in\PG$ is non-degenerate and spread-out and $N$ acts on a Hilbert space $H$,
then $\mathcal{H}_\alpha^\mu\,\cap\,\mathcal{K}(H) \,=\, \C1$.
\end{theorem}

\begin{proof}
Let $x\in \mathcal{H}_\alpha^\mu\cap\mathcal{K}(H)$ be self-adjonit. Then
there exists a state $\om\in B(H)_*$ such that $|\langle \omega , x \rangle| = \lVert x \rVert$. Hence the
restriction of $\omega$ to $N$ is a normal state on $N$ that realizes $\|x\|$.
Therefore $x\in \C1$ by Lemma \ref{lemma}.
Since $\mathcal{H}_\alpha^\mu\,\cap\,\mathcal{K}(H)$ is generated by its self-adjonit elements,
the result follows.
\end{proof}

\begin{corollary}
Let $\alpha : \G \curvearrowright N$ be an ergodic action of a locally compact quantum group $\g$ on a finite dimensional
von Neumann algebra $N$. Then $\h_\alpha^\mu = \C1$
for any non-degenerate spread-out $\mu\in\PG$.
\end{corollary}

In the case of the action of a locally compact (quantum) group $\G$
on itself via the comultiplication, there is always a natural
action of $\G$ on its Poisson boundaries. This follows from the fact that the comultiplication can be regarded both
as a left or a right action of $\G$ on itself that commute with each other. As shown
by the following theorem, in the general setting of quantum group actions, where such two sided commuting actions
do not exist, we cannot define a natural action of $\G$ on the corresponding Poisson boundaries.

\begin{theorem}\label{act}
Let $\alpha$ be an ergodic action of a locally compact quantum group $\G$ on a von Neumann algebra $N$,
and let $\mu\in\PG$ be non-degenerate and spread-out.
If $(\om\otimes\id)\alpha(x) \in \h_\alpha^\mu$ for every $\om\in\LO$, then $\h_\alpha^\mu = \C1$.
\end{theorem}

\begin{proof}
As in the proof of Proposition \ref{thm1} we conclude that
$\alpha(x) \in \LL\otimes_{\mathcal{F}} \h_\alpha^\mu = \LL\vtp \h_\alpha^\mu$,
for all $x\in\h_\alpha^\mu$. Hence the restriction of $\alpha$ to $\h_\alpha^\mu$
induces an ergodic action $\beta: \G \curvearrowright \h_\alpha^\mu$.
Therefore $\Phi_\beta^\mu = \Phi_\alpha^\mu \big|_{\h_\alpha^\mu}$,
which implies that $\h_\beta^\mu = \h_\alpha^\mu$. Hence $\h_\alpha^\mu = \C1$, by Corollary \ref{cor1}.
\end{proof}


\begin{corollary}
Let $\G$ be a locally compact quantum group. If there exists a non-degenerate spread-out $\mu\in\PG$
in the center of the  Banach algebra $\MG$, then $\h_\alpha^\mu = \C1$, and $\G$ is amenable.
\end{corollary}

\begin{proof}
If $\mu$ is in the center of $\MG$, then for every $x\in\h_\alpha^\mu$ and $\om\in\LO$,
\[
(\om\otimes\id)\,\alpha(x) \,=\, \Phi_\alpha^\om (x) \,=\, \Phi_\alpha^\om\, \Phi_\alpha^\mu(x) \,=\,
\Phi_\alpha^{\om\star\mu}(x)\,=\, \Phi_\alpha^{\mu\star\om}(x)\,=\,\Phi_\alpha^\mu\, \Phi_\alpha^\om(x) \,=\,
\Phi_\alpha^\mu\big((\om\otimes\id)\,\alpha(x)\big)
\]
which implies that $(\om\otimes\id)\alpha(x) \in \h_\alpha^\mu$,
for all $\om\in\LO$, hence $\h_\alpha^\mu = \C1$, by Theorem \ref{act}.
Also $\G$ is amenable by \cite[Theorem 4.2]{KNR}.
%
\end{proof}

The next corollary can be regarded as a generalization of
\cite [Theorem 4.3]{Kai-Ver}, \cite[Theorem 1.10]{Ros}, and \cite[Theorem 4.2]{KNR}.

\begin{corollary}
A locally compact quantum group $\G$ is amenable if and only if there exist a non-degenerate
spread-out $\mu\in\PG$ and $\nu\in\MG$ such that the spaces of fixed points of the left convolution
map by $\mu$ and the right convolution map by $\nu$ on $\LL$ coincide.
\end{corollary}

As another immediate corollary to Theorem \ref{act}, we obtain one of
the main results of \cite{Chu-Lau}, a {\it dual Choquet--Deny theorem}:
if $\G$ is a co-commutative locally compact quantum groups, then $\h_\Gamma^\mu = \C1$
for every non-degenerate $\mu \in \PG$.

Next we restrict ourselves to the actions of ``Compact Type".
A quantum version of the Choquet--Deny theorem for compact groups is proved in \cite[Theorem 5.3]{KNR}.
In the following, we generalize this result to the case of quantum ergodic actions.

Translating to our setting, the Choquet--Deny theorem for compact (quantum) groups
states the triviality of harmonic functions associated to the action
of a compact (quantum) group on itself, which is a ``\emph{compact}'' (quantum) space.
We show that weaker compactness type conditions are enough for the triviality of $\h_\alpha^\mu$.
First note that by \cite[Theorem 5.3]{KNR} and Corollary \ref{cor},
if $\G$ is compact and $\alpha : \G \curvearrowright N$ is ergodic,
then $\h_\alpha^\mu = \C 1$ for all non-degenerate $\mu\in\PG$.

A state $\varOmega$ on $N$ is said to be $\G$-invariant if $(\om \otimes\varOmega) \alpha = \langle \om\,,\,1\rangle\, \varOmega$
for all $\om\in\LO$.
Note that for a normal state $\nu\in N_*$, the $\G$-invariance means
$\om\phi_\nu = (\Phi_\alpha^\om)_*(\nu) = \langle \om\,,\,1\rangle\, \nu$
for all $\om\in\LO$.

\begin{theorem}\label{thm7}
Let $\alpha$ be an ergodic action of a locally compact quantum group $\g$ on a $\sigma$-finite
von Neumann algebra $N$. If there exists a normal $\G$-invariant state on $N$, then
$\h_\alpha^\mu = \C1$ for all non-degenerate spread-out $\mu\in\PG$.
\end{theorem}
\begin{proof}
Suppose that $\nu_0\in N_*$ is a normal $\G$-invariant state on $N$.
Further suppose that $\om$ is a faithful normal state on $\LL$.
Then it follows from Theorem \ref{lem} that
$\nu_0 \,=\, \om\,\phi_{\nu_0}$
is faithful on $N$. Now, if $\mu\in\PG$ is non-degenerate, it follows from (\ref{new1}) that $\nu_0\, E_\alpha^\mu \,=\, \nu_0$.
Hence, the completely positive map $E_\alpha^\mu$ is normal, and therefore $\h_\alpha^\mu = \C1$,
by Corollary \ref{cor1}.
\end{proof}

As a consequence, Theorem \ref{thm7} implies that harmonic
function are trivial in the case of probability measure preserving actions of locally compact groups.

We finish with another immediate corollary of Theorem \ref{thm7}, in the setting of group actions on finite factors.
\begin{corollary}
 Let $\alpha : G \curvearrowright N$ be an ergodic action of a locally
 compact group $G$ on a finite factor $N$, and $\mu$ be a non-degenerate
 probability measure on $G$. If $x\in N$ is such that
 \[
  \int_G\, \alpha_g\,(x)\, d\mu(g) \,=\, x\,,
 \]
 then $x\in\C 1$.
\end{corollary}

\begin{proof}
 One only needs to observe that the unique finite faithful trace $\tau$ on $N$
 is $G$-invariant.
 \end{proof}

\end{document}